\documentclass[12pt,a4paper]{amsart} 
\calclayout
\allowdisplaybreaks
\usepackage{amssymb}

\usepackage{mathrsfs} 
\usepackage{comment} 

\usepackage
{hyperref}

\usepackage[capitalise]{cleveref}
\usepackage{amsrefs}
\usepackage{xcolor}
\usepackage{enumerate}
\usepackage{autonum}

\theoremstyle{plain}
\newtheorem{theorem}{Theorem}[section]
\newtheorem{lemma}[theorem]{Lemma}

\newtheorem{corollary}[theorem]{Corollary}
\theoremstyle{definition}

\theoremstyle{remark}

\newcounter{statement}
\newenvironment{statement}
{
	
	\setcounter{statement}
	{\value{equation}}
	\begin{list}
		{(\theequation)}
		{
		}
		\usecounter{equation}
		\setcounter{equation}
		{\value{statement}}
		
		\item 
	}
	{\end{list}}

\title{On the spectral radius of the $(L,\kappa)$-lazy Markov chain}

\author[L. Qian]{Li Qian} 
\address[L. Qian]{School of Mathematical Sciences \\ Peking University \\ Beijing 100871 \\ China}
\email{ql1995@pku.edu.cn}

\author[Z. Sun]{Zhenyao Sun}
\address[Z. Sun]{School of Mathematics and Statistics \\ Beijing Institute of Technology \\ Beijing 100081 \\ China}
\email{zhenyao.sun@gmail.com}


\begin{document}
	\begin{abstract}
		We consider an $(L,\kappa)$-lazy operation on an irreducible Markov transition
		probability $P$ with state space $S$ where $L \subset S$ and $\kappa\in[0,1)$. For
		each $x \in L$ and $y\in S$, this $(L,\kappa)$-operation replaces $P(x,y)$, the
		transition probability from $x$ to $y$, by $\kappa 1_{\{x=y\}} + (1-\kappa)P(x,y)$. We
		are interested in how $L$ and $\kappa$ influence the spectral radius $\rho^L_\kappa$
		of this new transition probability. We first show that $\rho^L_\kappa$ is non-decreasing
		and continuous in $\kappa$. We then show that: (1) If $L$ is nonempty and finite, then
		$P$ being rho-transient is equivalent to that the growth of $\displaystyle
		(\rho^L_\kappa)_{\kappa\in[0,1)}$ exhibits a phase transition: There exists a critical
		value $\kappa_c(L) \in (0,1)$ such that $\kappa \mapsto \rho^L_\kappa$ is a constant
		on $[0,\kappa_c(L)]$ and increases strictly on $[\kappa_c(L),1)$; (2) For every
		$\kappa\in(0,1)$, if $S\setminus L$ is nonempty and finite, then
		$\rho^L_\kappa=\rho^S_\kappa$ if and only if $P$ is not strictly rho-recurrent.
	\end{abstract}
	
	\maketitle	
	
	\section{Introduction}
	We consider irreducible discrete-time Markov chains	on a countable infinite state space $S$.
	In recent years, there are some papers studied the $\kappa$-lazy version of a Markov transition probability $P$ on $S$,
	which is defined as $\kappa I + (1-\kappa)P$ where $I$ is the identity matrix (see \cites{FW21,LP17,MT06,H16}).
	Intuitively speaking, the corresponding $\kappa$-lazy Markov chain is the original chain being delayed at each step by tossing a coin with heads' probability $\kappa$: 
	If the chain gets a head, then it won't move, and it gets a tail, it will move according to the transition probability $P$.
	
    In this note, we are going to consider the $(L,\kappa)$-laziness which generalized the notion of the $\kappa$-laziness by only delaying the Markov chain on a subset of $S$. 
	To be more precise, for any $L\subset S$ and $\kappa \in [0,1)$, we denote by $P^L_\kappa$ the	
	\emph{$(L,\kappa)$-lazy version} of a transition probability $P$ which is defined so that for any $x,y \in S,$ 
	\begin{align}
		P^L_{\kappa}(x,y) :=
		\left\{
		\begin{aligned}
			& \kappa + (1-\kappa) P(x,x), && \text{ if } x \in L \text{ and } x= y;
			\\ & (1-\kappa) P(x,y), && \text{ if } x \in L \text{ and } x\neq y;
			\\ & P(x,y), && \text{ if } x\notin L.
		\end{aligned}
		\right.
	\end{align}
	We call $P (=P^L_0)$ the \emph {underlying transition probability}.
	
	We are particularly interested in how the spectral radius of the transition probability $P^L_\kappa$ is influenced by $L$ and $\kappa$.
	The \emph{Green function} $G(x,y|z)$ of a transition probability $P$ is defined as 
	\begin{align}
		G(x,y|z) := \sum_{n=0}^\infty \mathbb P^x(X_n=y) z^n, \quad x,y\in S, z\in \mathbb R,
	\end{align}
	where $\displaystyle (X_n)^\infty_{n=0}$ is a Markov chain with transition probability $P$ and initial value $X_0=x$ under the probability $\mathbb P^x$.
	It is proved in \cite{W00}*{Lemma 1.7} that if $P$ is irreducible,
	then the convergence radius of the Green function $G(x,y|\cdot)$ is independent of $x$ and $y$.
	The reciprocal of this convergence radius is known 
	as the \emph{spectral radius} of 
	$P$ (see \cite{W00} for example).
	It is noted in \cite{W00}*{(1.8)} that
	\begin{statement}\label{:rho<=1}
		the spectral radius of an irreducible transition probability takes its value in $(0,1]$.
	\end{statement}
	Observe that for $\kappa\in[0,1)$, if $P$ is irreducible, then so is $P^L_\kappa$.
	
	In the rest of this note, we will always assume that the underlying transition probability $P$ is irreducible, and we will denote by $\rho^L_\kappa$ the spectral radius of $P^L_\kappa$ with $\rho:= \rho^L_0$. 
	One can iterate the lazy operations on the underlying transition probability. 
	For example, we denote by $(P^L_\kappa)^{L'}_{\kappa'}$ the $(L',\kappa')$-lazy version of the $(L,\kappa)$-lazy version of $P$, and by $(\rho^L_\kappa)^{L'}_{\kappa'}$ the spectral radius of $(P^L_\kappa)^{L'}_{\kappa'}$.

	The spectral radius $\rho^S_\kappa$ for the $\kappa$-lazy version of $P$ has been studied in \cite{W00}*{Lemma 9.2}: For $\kappa\in[0,1)$, it holds that 
	\begin{align}\label{:glLi}
		\rho^S_\kappa = \kappa + (1-\kappa) \rho.
	\end{align}
	Our first result concerns the monotonicity and the continuity of $\kappa\mapsto \rho_\kappa^L$ for general $L \subset S$.
	\begin{theorem}\label{:gen}
		Let $L\subset S$ and $\kappa\in[0,1)$.
		\begin{enumerate}[(i)]
			\item \label{:gen:monotonicity}
			$\rho^L_{\kappa}$ is non-decreasing in $\kappa$.
			\item \label{:gen:continuity}
			$\rho^L_{\kappa}$ is continuous in $\kappa$.
			\item \label{:gen:1}
			$\rho=1 \Leftrightarrow \rho^L_\kappa=1$, or equivalently speaking, $\rho<1 \Leftrightarrow \rho^L_\kappa<1$.
			\item \label{:gen:lim1}
			If $L\neq \varnothing$, then $\lim_{\kappa \uparrow 1} \rho^L_\kappa = 1$.
		\end{enumerate}
	\end{theorem}
	
	Let us now introduce a classification of the transition probability,
	which is crucial for the rest of our results.
	It is proved in \cite{W00}*{Lemma 1.7} that $G(x,y|1/\rho)$ either $=\infty$ 
	(or $<\infty$)  
	simultaneously for all $x,y\in S$, and the corresponding $P$ is referred to as 
	\emph {rho-recurrent} (or \emph{rho-transient}) transition probability. 
	These concepts appeared in \cite{V62} and were studied in \cites{P64,P65,V67,V68,K76} and \cite{N04}*{Section 3.2}.
	One can find specific examples of rho-recurrent/rho-transient Markov chains in \cite{W00}*{Section 7.B}.
	
	\begin{theorem}\label{:IS?}
		Suppose $\rho<1$. 
		\begin{enumerate}[(i)]
			\item \label{:IS?:critical}
			If $0<\texttt{\#} L <\infty$ and $P$ is rho-transient, then there exists a unique $\kappa_{c}(L)\in (0,1)$ such that
			\begin{itemize}
				\item For $\kappa \in[0,\kappa_{c}(L))$,
				$\rho^L_\kappa = \rho$
				and $P^L_\kappa$ is rho-transient;
				\item For $\kappa \in [\kappa_{c}(L),1)$, 
				$\rho^L_\kappa$ increases strictly in $\kappa$ with $\rho^L_{\kappa_{c}(L)} = \rho$, 
				and $P^L_\kappa$ is rho-recurrent.
			\end{itemize}
			\item \label{:IS?:finiteIS}
			If $0<\texttt{\#} L <\infty$ and $P$ is rho-recurrent, then for every $\kappa\in[0,1)$, $\rho^L_\kappa$ increases strictly in $\kappa$ and $P^L_\kappa$ is rho-recurrent.
			\item \label{:IS?:comFiniteIS}
			If $\texttt{\#} (S\setminus L) <\infty$, then $\rho^L_\kappa$ increases strictly in $\kappa\in[0,1)$.
		\end{enumerate}
	\end{theorem}
	
	In this note, we will introduce a further classification for the rho-recurrent transition probabilities,
	    which we will use in our next result. 
	Let the \emph{U-function} of a transition probability $P$ be the power series:
	$$U(x,y|z):= \sum_{n=0}^\infty \mathbb P^x(\tau^y=n) z^n, \quad x,y\in S, z\in \mathbb R, $$
	where $\tau^y:= \inf\{ n\ge1: X_n = y\}$.
	Let $r(U|x,y)$ be the convergence radius of $U(x,y|\cdot)$. (Note that $r(U|x,y)$ may depend on $x,y$.)
	We will prove in \cref{:P:rhoTransient} that if $P$ is rho-transient, then
	\begin{align}\label{:rU=1/rho}
		r(U|x,x) = 1/\rho, \forall x\in S. 
	\end{align}
	If $P$ is rho-recurrent and \eqref{:rU=1/rho} holds, then we say $P$ is \emph{critically rho-recurrent}. 
	If $P$ is rho-recurrent, but \eqref{:rU=1/rho} does not hold, then we say $P$ is \emph{strictly rho-recurrent}. 
	
	\begin{theorem}\label{:twoRhoRecurrences}
		Suppose that $\rho<1$ and $0<\texttt{\#} L <\infty$.
		\begin{enumerate}[(i)]
			\item If $P$ is rho-transient, then $P^L_{\kappa_{c}(L)}$ is critically rho-recurrent where $\kappa_{c}(L)$ is given as in \cref{:IS?} \eqref{:IS?:critical}.
			Moreover, for every $\kappa \in (\kappa_{c}(L),1)$, $P^L_\kappa$ is strictly rho-recurrent.
			\item If $P$ is rho-recurrent, then for every $\kappa\in(0,1)$, $P^L_\kappa$ is strictly rho-recurrent.
		\end{enumerate}
	\end{theorem}		 
	
	For $\kappa\in(0,1)$, 
	observe that $\displaystyle(P^L_\kappa)^{S\setminus L}_\kappa = P^S_\kappa$.
	Hence by \cref{:gen} \eqref{:gen:monotonicity}, $\rho^L_\kappa \le \rho^S_\kappa$.
	
	\begin{theorem}\label{:reach?}
		Suppose that $\rho<1$, $\texttt{\#} (S\setminus L) >0$ and $\kappa\in(0,1)$.
		\begin{enumerate}[(i)]
			\item \label{:reach?:invariant}
			If $\texttt{\#} (S\setminus L) <\infty$, then 
			$\rho^L_\kappa = \rho^S_\kappa$ 
			if and only if $P$ is not strictly rho-recurrent.
			\item \label{:reach?:strictlyRhoRecurrentUnder} 
			If $P$ is strictly rho-recurrent, then $\rho^L_\kappa<\rho^S_\kappa$. 
			\item \label{:reach?:finiteUnder}
			If $\texttt{\#} L <\infty$, then $\rho^L_\kappa<\rho^S_\kappa$.
		\end{enumerate}
	\end{theorem}
	
	The rest of this note is organized as follows. 
	\cref{:P} gives some preliminary results that we will use throughout the note.
	Proof of \cref{:gen} is provided in \cref{:gen:pf}.
	In \cref{:pf1}, we show the proofs of \cref{:IS?} \eqref{:IS?:critical}, \eqref{:IS?:finiteIS}, and \cref{:reach?} \eqref{:reach?:finiteUnder}.
	In \cref{:pf2}, we give the proofs of \cref{:twoRhoRecurrences}, \cref{:reach?} \eqref{:reach?:invariant}, \eqref{:reach?:strictlyRhoRecurrentUnder}, and \cref{:IS?} \eqref{:IS?:comFiniteIS}.
	
	\subsection*{Acknowledgment}
	Part of this research was done while the second author was a Postdoc at the Technion—Israel Institute of Technology, supported by a scholarship from the Israel Council for Higher Education.
	
	The authors want to thank Dayue Chen for helpful conversations. 
	
	\section{Preliminary}\label{:P}
	This section will introduce some basic results for irreducible Markov chains.
	Recall that we used notations $(\mathbb P^x)_{x\in S}, G(\cdot,\cdot|\cdot), \rho, U(\cdot,\cdot|\cdot)$, 	and $r(U|\cdot,\cdot)$ to represent the probability of a Markov chain, the Green function, the spectral radius, the U-function, and the convergence radius of the U-function, corresponding to an irreducible transition probability $P$ on $S$, respectively.
	In the rest of this note, we will use notations $(\mathbb P^{L,x}_\kappa)_{x\in S}, G^L_\kappa(\cdot,\cdot|\cdot), \rho^L_\kappa, U^L_\kappa(\cdot,\cdot|\cdot)$ and $r(U^L_\kappa|\cdot,\cdot)$ to represent the similar concepts for $P^L_\kappa$, the $(L,\kappa)$-lazy version of $P$.
	
	\subsection{Markov chains}\label{:P:MC}
	We first observe that for any $x\in S$, 
	\begin{align}\label{:P:rG<=rU}
		r(U|x,x) \ge 1/\rho
	\end{align}
	by the Cauchy-Hadamard formula.
	It is then easy to see that $P$ is strictly rho-recurrent if and only if
	\begin{align}\label{:P:SCR>}
		\exists x\in S, r(U|x,x) > 1/\rho.
	\end{align}
	The following lemma connects the Green function, the U-function, and the spectral radius. 
	Some of these results were known in the literature (see \cite{W00}*{Lemma 1.13} for examples).
	Here, we include their proofs for the sake of completeness.
	
	\begin{lemma}\label{:P:GU}
		For $x\in S$ and $z>0$, 
		\begin{enumerate}[(i)]
			\item \label{step1}
			$U(x,x|z)<1 \Leftrightarrow G(x,x|z)<\infty.$
			\item \label{step2}
			If $G(x,x|z)<\infty$, then
			\begin{equation}\label{GU=}
				G(x,x|z) = \frac{1}{1-U(x,x|z)}.
			\end{equation}
			\item \label{step3} 
			$U(x,x|z)>1 \Leftrightarrow z>1/\rho.$
			\item \label{step4}
			$1/\rho=\max\{z>0: U(x,x|z) \le 1 \}.$
		\end{enumerate}
	\end{lemma}
	
	\begin{proof}
		For $N\in \mathbb Z^+$, 
		\begin{align} \label{eqGU}
			\sum_{n=0}^N \mathbb P^x(X_n=x) z^n
			&= 1 + \sum_{n=1}^N \sum_{m=0}^n \mathbb P^x(\tau^x=m) \mathbb P^x(X_{n-m}=x) z^n
			\\&= 1 + \sum_{m=0}^N \mathbb P^x(\tau^x=m) z^m \sum_{n=m}^N \mathbb P^x(X_{n-m}=x) z^{n-m}
			\\&= 1 + \sum_{m=0}^N \mathbb P^x(\tau^x=m) z^m \sum_{n=0}^{N-m} \mathbb P^x(X_n=x) z^{n}
			.
		\end{align}
		Note that 
		\begin{align}\label{eqGU1}
			& \sum_{n=0}^N \mathbb P^x(X_n=x) z^n
			\ge 1+ \sum_{m=0}^{\lfloor N/2\rfloor} \mathbb P^x(\tau^x=m) z^m \sum_{n=0}^{\lfloor N/2\rfloor} \mathbb P^x(X_n=x) z^n.
		\end{align}	    
		If $G(x,x|z)<\infty$, then by taking $N$ to infinity, we have $U(x,x|z)<1$.
		By taking $N$ to infinity in
		both
		\begin{align}\label{eqGU2}
			& \sum_{n=0}^N \mathbb P^x(X_n=x) z^n
			\le 1+ \sum_{m=0}^N \mathbb P^x(\tau^x=m) z^m \sum_{n=0}^N \mathbb P^x(X_n=x) z^n
		\end{align}
		and \eqref{eqGU1}, we have \eqref{step2}.
		
		Now we assume that $G(x,x|z)=\infty$.
		By \eqref{eqGU2}, we have 
		$$\sum_{m=0}^N \mathbb P^x(\tau^x=m) z^m \ge \frac{\sum_{n=0}^N \mathbb P^x(X_n=x) z^n-1}{\sum_{n=0}^N \mathbb P^x(X_n=x) z^n}.$$
		Let $N\to\infty$, we have $U(x,x|z) \ge 1$.
		Therefore \eqref{step1} is proved.
		
		For $\Rightarrow$ of \eqref{step3}, by the strictly increasing property of $U(x,x|\cdot)$, it suffices to show that $U(x,x|1/\rho) \le 1$. 
		Noticing that the coefficients of the power series $U(x,x|\cdot)$ is nonnegative, 
		by the monotone convergence theorem, we have 
		$$\lim\limits_{z\uparrow 1/\rho} U(x,x|z) = U(x,x|1/\rho).$$
		By \eqref{step1} we have that $U(x,x|z)<1$ for $z<1/\rho$. 
		Therefore $U(x,x|1/\rho) \le1$ as desired.
		
		For $\Leftarrow$ of \eqref{step3}, $z>1/\rho$ implies that $G(x,x|z)=\infty$, 
		and, by \eqref{step1}, further implies that $U(x,x|z)\ge1$. We only have to exclude $U(x,x|z)=1$ by contradiction: If it holds, then for $w\in(1/\rho,z)$, $U(x,x|w)<1$, which by \eqref{step1}, contradicts the fact that $G(x,x|w)<\infty$.  
		
		\eqref{step4} can be directly concluded from \eqref{step3}.
	\end{proof}
	
	As a corollary, we have another equivalent condition for the rho-recurrence and the rho-transience.
	\begin{corollary}\label{:P:U(1/rho)?1}
		For any $x\in S$, $P$ is rho-recurrent $\Leftrightarrow U(x,x|1/\rho) =1$;
		and $P$ is rho-transient $\Leftrightarrow U(x,x|1/\rho) <1$.
	\end{corollary}
	\begin{proof}
		From \cref{:P:GU} \eqref{step1}, we only have to prove $\Rightarrow$ of the first statement. 
		When $P$ is rho-recurrent, by \cref{:P:GU} \eqref{step1} we know that $U(x,x|1/\rho) \ge 1$.
		By \cref{:P:GU} \eqref{step3} we have $U(x,x|1/\rho)\le1$.
		We are done.
	\end{proof}

	\begin{lemma}\label{:P:rhoTransient}
		If $P$ is rho-transient, then $r(U|x,x) = 1/\rho$ for every $x\in S$.
	\end{lemma}
	
	\begin{proof}
		For the sake of contradiction and \eqref{:P:rG<=rU}, we assume that there exists $x\in S$ such that $r(U|x,x) > 1/\rho$. 
		Then by the continuity of the power series inside of its convergence radius and \cref{:P:U(1/rho)?1}, 
		there exists $z > 1/\rho$ such that $U(x,x|z)<1$.
		Now by \cref{:P:GU} \eqref{step1}, 
		we have $G(x,x|z)<\infty$
		which contradicts the fact that $1/\rho$ is the convergence radius of $G(x,x|\cdot)$.
	\end{proof}
	
	\subsection{\texorpdfstring{$(L,\kappa)$-laziness}{}}\label{:P:ULL}
	
	By \cref{:P:GU}, the U-function is a good tool for studying the spectral radius.
	
	\begin{lemma}\label{:P:UkappaExpansion}
		Let $x\in S, \kappa\in[0,1)$ and $z \ge0$.
		\begin{enumerate}[(i)]
			\item \label{:P:UkappaExpansion:single}
			If $L=\{x\}$, then
			\begin{align}
				U^L_\kappa(x,x|z) = \kappa z + (1-\kappa) U(x,x|z).
			\end{align}
			\item \label{:P:UkappaExpansion:general}
			If $x\notin L$, 
			$L \neq \varnothing$ 
			and $\kappa z<1$, then 
			\begin{align}
				U^L_\kappa(x,x|z) =  
				\mathbb P^{x}(\tau^x=1) z + \sum_{k \ge2, \vec l \in \mathcal{T}(k,x)} P(\vec l) z^k
				\left(\frac{1-\kappa}{1-\kappa z}\right)^{\texttt{\#} \vec l(L)}
				,
			\end{align}
			where for $k\ge 2$,
			\begin{align}
				\mathcal{T}(k,x) := 
				\{ &(x_0,x_1,x_2,\cdots,x_{k-1},x_k) \in S^{k+1} : 
				\\ & x_0=x; 
				x_i \neq x, \forall i=1,\cdots,k-1; 
				x_k=x 
				\},
			\end{align}
			and for $\vec l = (x_0,x_1,x_2,\cdots,x_{k-1},x_k) \in \mathcal{T}(k,x) $,  
			$$P(\vec l):= \prod_{i=0}^{k-1} P(x_i,x_{i+1}),\ 
			\texttt{\#} \vec l(L) := \texttt{\#} \{ i: i =1,\cdots,k-1; x_i\in L \}.$$
			\item \label{:P:UkappaExpansion:infinity}
			If $x\notin L$, 
			$L \neq \varnothing$ 
			and $\kappa z\ge1$, then 
			$
			U^L_\kappa(x,x|z) = \infty.
			$
		\end{enumerate}
	\end{lemma}
	\begin{proof}
		For \eqref{:P:UkappaExpansion:single}, it is done by the following:
		\begin{align}
			\mathbb P^{\{x\},x}_\kappa(\tau^x=n) 
			= 
			\left \{ 
			\begin{aligned}
				& \kappa + (1-\kappa) \mathbb P^{x}(\tau^x=1) 
				,&&  \text{ if } n=1;
				\\ & (1-\kappa) \mathbb P^{x}(\tau^x=n) 
				,&&  \text{ if } n\ge2. 
			\end{aligned}
			\right.
		\end{align}
		
		For \eqref{:P:UkappaExpansion:general} and \eqref{:P:UkappaExpansion:infinity}, noticing that $x\notin L$ and $L\neq \varnothing$,
		we assert that
		\begin{align}\label{:P:UkappaExpansion:PExpansion}
			\mathbb P^{L,x}_\kappa(\tau^x=n) 
			= 
			\left \{ 
			\begin{aligned}
				& \mathbb P^{x}(\tau^x=1) 
				,&&  \text{ if } n=1;
				\\ & 
				\sum_{k=2}^n \sum_{\vec l \in \mathcal{T}(k,x)} 
				P(\vec l) (1-\kappa)^{\texttt{\#} \vec l(L)} 
				(-\kappa)^{n-k}  
				\binom{ -\texttt{\#} \vec l(L)}{n-k} 
				,&& \text{ if } 			
				n\ge2,
			\end{aligned}
			\right.
		\end{align}
		where we used the generalized binomial series: For $k \in \mathbb {Z}^+$ and arbitrary $\alpha\in \mathbb R$, 
		$$\displaystyle {\binom {\alpha }{k}}
		:={\frac {\alpha (\alpha -1)(\alpha -2)\cdots (\alpha -k+1)}{k!}}, \text{ in particular } \binom {\alpha }{0} = 1.$$
		
		Let us explain \eqref{:P:UkappaExpansion:PExpansion} when $n\ge2$.
		A Markov chain $X_t$ with transition probability $P^L_\kappa$ can be constructed in the following way:
		At each time $t\in \mathbb Z^+$, taking an independent uniform r.v. $\theta_t$ in $[0,1]$, if $X_{t-1}\in L$ and $\theta_t \le \kappa$, then we set $X_t:=X_{t-1}$ and say that the Markov chain takes a lazy step at time $t$;
		else if $X_{t-1}\notin L$ or $\theta_t >\kappa$, we sample $X_t$ according to the probability $\{P(X_{t-1},y): y\in S\}$ 
		and say that the chain takes a non-lazy step.
		
		We refer to \emph{the excursion} the trajectory of the Markov chain $(X_t)$ up to the time $\tau^x$.
		We refer to \emph{the non-lazy excursion} the trajectory of the Markov chain $X_t$ formed only by the non-lazy steps up to the time $\tau^x$.
		It is observed that on the event $\{\tau^x=n\}$, the excursion takes its value in $\mathcal T(n,x)$ while the non-lazy excursion takes its value in $\cup_{2\le k\le n} \mathcal T(k,x)$.
		
		For a given $2\le k\le n$ and $l\in \mathcal T(k,x)$, by the elementary combinatorics, the number of possible excursions in the event $\{\tau^x=n, l \text{ is the non-lazy excursion}\}$ is $\binom{-\texttt{\#} \vec l(L)}{n-k} (-1)^{n-k}$,
		and each of those excursions happens with the same probability $P(\vec l) (1-\kappa)^{\texttt{\#} \vec l(L)} \kappa^{n-k}$.
		Therefore $\mathbb P^x(\tau^x=n, l \text{ is the non-lazy excursion}) = P(\vec l) (1-\kappa)^{\texttt{\#} \vec l(L)} (-\kappa)^{n-k} \binom{-\texttt{\#} \vec l(L)}{n-k}$.
		Now \eqref{:P:UkappaExpansion:PExpansion} holds.
		
		By Fubini's theorem for nonnegative series ($\binom{-\texttt{\#} \vec l(L)}{n-k} (-1)^{n-k}$ is nonnegative), we have that
		\begin{align} 
			U^L_\kappa(x,x|z)
			&= \mathbb P^{x}(\tau^x=1) z
			+ \sum_{n\ge2} \sum_{k=2}^n \sum_{\vec l \in \mathcal{T}(k,x)} P(\vec l) (1-\kappa)^{\texttt{\#} \vec l(L)}  (-\kappa)^{n-k}  \binom{-\texttt{\#} \vec l(L)}{n-k}  z^n
			\\ &=  \mathbb P^{x}(\tau^x=1) z
			+ \sum_{k \ge2, \vec l \in \mathcal{T}(k,x)} P(\vec l) (1-\kappa)^{\texttt{\#} \vec l(L)} z^k
			\sum_{n\ge k}  \binom{-\texttt{\#} \vec l(L)}{n-k} (-\kappa z)^{n-k}.
		\end{align}
		For $\kappa z<1$, by the generalized binomial theorem, we have that
		\begin{align}
			\sum_{n\ge k}  \binom{-\texttt{\#} \vec l(L)}{n-k} (-\kappa z)^{n-k}
			= (1-\kappa z)^{-\texttt{\#} \vec l(L)}.
		\end{align}
		Now we have \eqref{:P:UkappaExpansion:general}.
		
		If $\kappa z\ge1$, then by $L\neq \varnothing$ and the irreducibility, there exist $k_0 \ge2$ and $\vec l_0 \in \mathcal{T}(k_0,x)$ such that $P(\vec l_0)>0$ and $\texttt{\#} \vec l_0(L) >0$.
		As $\kappa<1$ and $z>0$, we have that
		\begin{align}
			P(\vec l_0) (1-\kappa)^{\texttt{\#} \vec l_0(L)} z^{k_0}
			\sum_{n\ge k_0}  \binom{-\texttt{\#} \vec l_0(L)}{n-k_0} (-\kappa z)^{n-k_0} = \infty.
		\end{align}
		Thus we have \eqref{:P:UkappaExpansion:infinity}.
	\end{proof}
	
	\begin{corollary}\label{:P:rU=rU}
		If $L=\{x\}$, then $r(U^L_\kappa|x,x) = r(U|x,x)$ for every $\kappa\in[0,1)$.
	\end{corollary}
	
	Let us state two more results when $L=S$.
	\begin{lemma}[\cite{W00}, Lemma 9.2]\label{:P:globallyUniformlyLazyGG}
		For $\kappa\in[0,1)$, $x\in S$ and $z\in[0,1/\rho^{S}_{\kappa})$,
		$$G^S_\kappa(x,x|z) = \frac{1}{1-\kappa z} G\left(x,x \middle| \frac{(1-\kappa)z}{1-\kappa z}\right).$$
	\end{lemma}
	
	\begin{lemma}\label{:P:globallyUniformlyLazyUU}
		For $\kappa\in[0,1)$, $x\in S$ and $z\in\left[0,r(U^S_\kappa|x,x)\right)$,
		$$U^S_\kappa(x,x|z) = (1-\kappa z) U\left(x,x\middle|\frac{(1-\kappa)z}{1-\kappa z}\right) +\kappa z,$$
		and
		$$1/r(U^S_\kappa|x,x) = \kappa + (1-\kappa)/r(U|x,x).$$
	\end{lemma}
	\begin{proof}
		It is similar to the proof of Lemma 9.2 in \cite{W00}, noticing by \eqref{:P:UkappaExpansion:PExpansion} that
		\begin{align}
			\mathbb P^{S,x}_\kappa(\tau_x=n) = \left\{
			\begin{aligned}
				&\kappa + (1-\kappa) \mathbb P^{x}(\tau_x=1), &n=1;
				\\ &\sum\limits_{k=2}^n \mathbb P^{x}(\tau_x=k) (-\kappa)^{n-k} (1-\kappa)^k \binom{-k+1}{n-k}, &n\ge2.
			\end{aligned}
			\right.
		\end{align}
	\end{proof}
	
	\begin{lemma}\label{:P:globallyUniformlyLazyMaintenance}
		For $\kappa\in[0,1)$, 
		$P$ is rho-transient (critically rho-recurrent, or strictly rho-recurrent, respectively), 
		if and only if so is $P^S_\kappa$. 
	\end{lemma}
	\begin{proof}
		By \eqref{:glLi} and \eqref{:rho<=1}, $\rho^S_\kappa > \kappa$, and thus $1-\kappa/\rho^S_\kappa>0$.
		By the monotone convergence theorem and \cref{:P:globallyUniformlyLazyGG},
		\begin{align}
			G^S_\kappa(x,x|1/\rho^S_\kappa)
			&=\lim_{z \uparrow 1/\rho^S_\kappa}
			G^S_\kappa(x,x|z)
			= \lim_{z \uparrow 1/\rho^S_\kappa}
			\frac{1}{1-\kappa z}
			G(x,x|\frac{(1-\kappa) z}{1-\kappa z})
			\\&=\frac{1}{1-\kappa/\rho^S_\kappa}
			G(x,x|1/\rho).
		\end{align}
		Therefore $P$ being rho-transient (or rho-recurrent) is equivalent to that $P^S_\kappa$ being rho-transient (or rho-recurrent).
		
		If $P$ is critically (or strictly) rho-recurrent, then for any $x\in S$, $r(U|x,x) = 1/\rho$ (or by \eqref{:P:SCR>}, there exists $x$ such that $r(U|x,x) > 1/\rho$).
		By \cref{:P:globallyUniformlyLazyUU} and \eqref{:glLi}, $r(U^S_\kappa|x,x)= 1/\rho^S_\kappa$ (or $r(U^S_\kappa|x,x) > 1/\rho^S_\kappa$).
		The other direction is the same.
		Now
		the proof is done. 
	\end{proof}
	
	\section{\texorpdfstring{Proof of \cref{:gen}}{}}\label{:gen:pf}
	\begin{proof}
		[Proof of \cref{:gen}] \eqref{:gen:monotonicity}.
		Given $0 \le \kappa <\kappa' < 1$, we want to prove $\rho^L_\kappa \le \rho^L_{\kappa'}$.
		If $L=\emptyset$, then it is trivial.
		If $L=S$, then by \eqref{:rho<=1} and \eqref{:glLi}, it is done. 
		We assume that $L\neq S$ and $L\neq \varnothing$, and choose $x\in S\setminus L$.
		
		We claim that:
		$\forall z \in [1,1/\kappa')$, $U^L_\kappa(x,x|z) \le U^L_{\kappa'}(x,x|z)$.
		In fact, as $x\notin L$ and $1>\kappa' z >\kappa z$, by \cref{:P:UkappaExpansion} \eqref{:P:UkappaExpansion:general}, both $U^L_\kappa(x,x|z)$ and $U^L_{\kappa'}(x,x|z)$ can be expanded. 
		As $1 > \kappa' z > \kappa z$, 
		we have $\displaystyle\frac{1-\kappa}{1-\kappa z} \le \frac{1-\kappa'}{1-\kappa' z}$.
		Therefore the claim holds by the above mentioned expansions.
		
		By \cref{:P:GU} \eqref{step3}, we have $U^L_{\kappa'}(x,x|1/\rho^L_{\kappa'}) \le1$.
		Then as $x\notin L$ and $L\neq \varnothing$, by \cref{:P:UkappaExpansion} \eqref{:P:UkappaExpansion:infinity}, we have $1/\rho^L_{\kappa'} < 1/{\kappa'}$.
		Therefore by the above claim, $U^L_\kappa(x,x|1/\rho^L_{\kappa'}) \le U^L_{\kappa'}(x,x|1/\rho^L_{\kappa'}) \le 1$.
		By \cref{:P:GU} \eqref{step3}, $1/\rho^L_{\kappa'} \le 1/\rho^L_{\kappa}$.
		
		\eqref{:gen:continuity}.
		Firstly, we are going to prove the right continuity.
		Set $0\le\kappa<\kappa' <1$.
		By the \cref{:gen} \eqref{:gen:monotonicity}, we have $\rho^L_{\kappa} \le \rho^L_{\kappa'}$. 
		By the sandwich theorem, it suffices to show that 
		\begin{enumerate}[(a)]
			\item \label{:rhoCo:R:1}
			$\rho^L_{\kappa'} \le (\rho^L_\kappa)^S_{{\kappa_1}(\kappa')}$;
			\item \label{:rhoCo:R:2}
			$\lim_{\kappa' \downarrow \kappa} (\rho^L_\kappa)^S_{{\kappa_1}(\kappa')} = \rho^L_\kappa$,
		\end{enumerate}
		where ${\kappa_1}(\kappa') := (\kappa' - \kappa)/(1-\kappa).$
		It is clear that \eqref{:rhoCo:R:2} follows from $\lim_{\kappa' \downarrow \kappa}{\kappa_1}(\kappa') = 0$ and \eqref{:glLi}.
		Noticing $\displaystyle (P^L_\kappa)^S_{{\kappa_1}} = (P^L_{\kappa'})^{S\setminus L}_{{\kappa_1}}$,
		and \cref{:gen} \eqref{:gen:monotonicity},
		we have \eqref{:rhoCo:R:1} by $\displaystyle (\rho^L_\kappa)^S_{{\kappa_1}} 
		= (\rho^L_{\kappa'})^{S\setminus L}_{{\kappa_1}}
		\ge \rho^L_{\kappa'}$.		
		
		Secondly, we prove the left continuity.
		For $0\le\kappa''<\kappa <1$, to show $\lim_{\kappa''\uparrow \kappa} \rho^L_{\kappa''} = \rho^L_{\kappa}$,
		as $\rho^L_{\kappa''} \le \rho^L_{\kappa}$, it suffices to show that
		$  -{\kappa_2} +(1+{\kappa_2})  \rho^L_\kappa \le \rho^L_{\kappa''}$ where ${\kappa_2}:=(\kappa-\kappa'')/(1-\kappa)$.
		By \eqref{:glLi}, it suffices to prove that 
		\begin{align}
			\rho^L_\kappa \le \frac{\rho^L_{\kappa''} + {\kappa_2}}{1+{\kappa_2}} 
			= (\rho^L_{\kappa''})^S_{\frac{{\kappa_2}}{1+{\kappa_2}}},
		\end{align}
		which can be derived from
		$\displaystyle (P^L_{\kappa''})^S_{\frac{{\kappa_2}}{1+{\kappa_2}}} 
		= (P^L_{\kappa})^{S\setminus L}_{\frac{{\kappa_2}}{1+{\kappa_2}}}$
		and \cref{:gen} \eqref{:gen:monotonicity}.
		
		\eqref{:gen:1}.
		If $\rho=1$, by \cref{:gen} \eqref{:gen:monotonicity}, $\rho^L_{\kappa} \ge 1$.
		By \eqref{:rho<=1}, we have $\rho^L_{\kappa} \le 1$. We obtain the conclusion.
		If $\rho^L_{\kappa}=1$, by \eqref{:glLi} and \cref{:gen} \eqref{:gen:monotonicity}, $\kappa + (1-\kappa)\rho = \rho^S_{\kappa} \ge \rho^L_{\kappa}=1$. Hence $\rho\ge1$. 
		From \eqref{:rho<=1}, we are done.
		
		\eqref{:gen:lim1}.
		As $L\neq\varnothing$, choose $x\in L$.
		By \cref{:gen} \eqref{:gen:monotonicity} and \eqref{:rho<=1}, $\displaystyle \rho^{\{x\}}_\kappa \le \rho^L_\kappa \le 1$.
		Then by the sandwich theorem it suffices to prove that $\lim_{\kappa\uparrow1}\rho^{\{x\}}_\kappa = 1$.
		By \eqref{:rho<=1}, it suffices to show that $\forall \epsilon \in(0,1), \exists K\in [0,1), \forall \kappa\in(K,1), 1-\rho^{\{x\}}_\kappa < \epsilon$.
		By \cref{:P:GU} \eqref{step3}, $1-\rho^{\{x\}}_\kappa < \epsilon$ is equivalent to $U^{\{x\}}_\kappa(x,x|1/(1-\epsilon)) >1$.
		By \cref{:P:UkappaExpansion} \eqref{:P:UkappaExpansion:single},
		$U^{\{x\}}_\kappa(x,x|1/(1-\epsilon)) = \kappa/(1-\epsilon) + (1-\kappa) U(x,x|1/(1-\epsilon))$.
		As $\lim_{\kappa\uparrow1} U^{\{x\}}_\kappa(x,x|1/(1-\epsilon)) = 1/(1-\epsilon) >1$, we can find $K$ to obtain the desired result.
	\end{proof}
	
	\section{\texorpdfstring{Proof of \cref{:IS?} \eqref{:IS?:critical}, \eqref{:IS?:finiteIS} and \cref{:reach?} \eqref{:reach?:finiteUnder}}{}}\label{:pf1}
	\begin{lemma}\label{:singleCri:3}
		Let $0< \texttt{\#}  L <\infty$ and $0 \le \kappa_1 <\kappa_2 < 1$,
		\begin{enumerate}[(i)]
			\item \label{:singleCri:3:12}
			If $P^L_{\kappa_2}$ is rho-transient, then $\rho= \rho^L_{\kappa_1} = \rho^L_{\kappa_2}$, and $P^L_{\kappa_1}$ is rho-transient. 
			As a consequence, if $P^L_{\kappa_1}$ is rho-recurrent, then $P^L_{\kappa_2}$ is rho-recurrent.
			\item \label{:singleCri:3:3}
			If $P^L_{\kappa_1}$ is rho-recurrent and $\rho^L_{\kappa_1}<1$, then $\rho^L_{\kappa_1} < \rho^L_{\kappa_2}$.
		\end{enumerate}
	\end{lemma}
	\begin{proof}
		We first show \eqref{:singleCri:3:12} under the condition that $L=\{x\}$.
		As $P^{\{x\}}_{\kappa_2}$ is rho-transient,
		by \eqref{:P:rG<=rU}, \cref{:gen} \eqref{:gen:monotonicity}, \cref{:P:rhoTransient} and \cref{:P:rU=rU}, we have
		\begin{align}
			r(U|x,x) \ge 1/\rho\ge 1/\rho^{\{x\}}_{\kappa_1} \ge 1/\rho^{\{x\}}_{\kappa_2} = r(U^{\{x\}}_{\kappa_2}|x,x)= r(U|x,x).
		\end{align}
		Then $\rho= \rho^{\{x\}}_{\kappa_1} = \rho^{\{x\}}_{\kappa_2}$.
		Thus, by using \cref{:P:GU} \eqref{step3} and \cref{:P:UkappaExpansion} \eqref{:P:UkappaExpansion:single}, we have that
		\begin{align}\label{:singleCri:3:expand}
			1 
			\ge U^{\{x\}}_{\kappa_1}(x,x|1/\rho^{\{x\}}_{\kappa_1})
			= U^{\{x\}}_{\kappa_1}(x,x|1/\rho^{\{x\}}_{\kappa_2}) 
			= \kappa_1 /\rho^{\{x\}}_{\kappa_2} + (1-\kappa_1) U(x,x|1/\rho^{\{x\}}_{\kappa_2}
			).
		\end{align}
		By \eqref{:rho<=1}, $1/\rho^{\{x\}}_{\kappa_2} \ge 1$.
		Then by \eqref{:singleCri:3:expand}, $U(x,x|1/\rho^{\{x\}}_{\kappa_2}) \le 1$.
		Thus, by \cref{:P:UkappaExpansion} \eqref{:P:UkappaExpansion:single} and \cref{:P:U(1/rho)?1}, we have that
		\begin{align}
			U^{\{x\}}_{\kappa_1}(x,x|1/\rho^{\{x\}}_{\kappa_1})
			\le \kappa_2 /\rho^{\{x\}}_{\kappa_2} + (1-\kappa_2) U(x,x|1/\rho^{\{x\}}_{\kappa_2})
			= U^{\{x\}}_{\kappa_2}(x,x|1/\rho^{\{x\}}_{\kappa_2}) < 1.
		\end{align}
		Then by \cref{:P:U(1/rho)?1}, $P^{\{x\}}_{\kappa_1}$ is also rho-transient as desired.
		
		We then show \eqref{:singleCri:3:3} under the condition that $L=\{x\}$.
		Using \cref{:P:UkappaExpansion} \eqref{:P:UkappaExpansion:single} twice and \cref{:P:U(1/rho)?1},
		\begin{align}
			U^{\{x\}}_{\kappa_2}(x,x|1/\rho^{\{x\}}_{\kappa_1}) 
			= \frac{\kappa_2-\kappa_1}{1-\kappa_1} / \rho^{\{x\}}_{\kappa_1} + \frac{1-\kappa_2}{1-\kappa_1} U^{\{x\}}_{\kappa_1}(x,x|1/\rho^{\{x\}}_{\kappa_1})
			= \frac{\kappa_2-\kappa_1}{1-\kappa_1} / \rho^{\{x\}}_{\kappa_1} + \frac{1-\kappa_2}{1-\kappa_1}.
		\end{align}
		By $\frac{1-\kappa_2}{1-\kappa_1} \in (0,1)$ and $1/\rho^{\{x\}}_{\kappa_1} >1$, we know that $U^{\{x\}}_{\kappa_2}(x,x|1/\rho^{\{x\}}_{\kappa_1}) >1$.
		By \cref{:P:GU} \eqref{step3}, we know that $1/\rho^{\{x\}}_{\kappa_1} > 1/\rho^{\{x\}}_{\kappa_2}$ as desired.
		
		Finally, let us show both \eqref{:singleCri:3:12} and \eqref{:singleCri:3:3} when $L=\{x_1,\cdots,x_m\}$ with $m\ge2$.
		In this case
		$\displaystyle P^L_{\kappa_2} = (\cdots((P^L_{\kappa_1})^{\{x_1\}}_\kappa)^{\{x_2\}}_\kappa\cdots)^{\{x_m\}}_\kappa$
		with $\kappa = (\kappa_2-\kappa_1)/(1-\kappa_1)$.
		It is clear that we can obtain the desired result of \eqref{:singleCri:3:12} by induction.
		For \eqref{:singleCri:3:3}, by what we have already proved and \cref{:gen} \eqref{:gen:monotonicity}, we have $\displaystyle \rho^L_{\kappa_1} < (\rho^L_{\kappa_1})^{\{x_1\}}_\kappa \le \rho^L_{\kappa_2}$ as desired.
	\end{proof}
	
	\begin{proof}[Proof of \cref{:IS?} \eqref{:IS?:critical}]
		{\it Step 1.}
		We only have to consider the existence part since the uniqueness is trivial.
		Define $\kappa_c(L) := \sup \{ \kappa\in[0,1): P^{L}_\kappa \text{ is rho-transient} \}$.
		As $P$ is rho-transient, we know $\kappa_c(L)\in[0,1]$.
		By \cref{:singleCri:3} \eqref{:singleCri:3:12}, if $\kappa_c(L)>0$ and $\kappa \in [0,\kappa_c(L))$, then $P^L_\kappa$ is rho-transient and $\rho = \rho^L_\kappa$;
		if $\kappa_c(L)<1$ and $\kappa \in (\kappa_c(L),1)$, then $P^L_\kappa$ is rho-recurrent. 
		Also in the latter case, since $\rho^L_\kappa<1$, by \cref{:gen} \eqref{:gen:1}	and \cref{:singleCri:3} \eqref{:singleCri:3:3}, 
		we have that $\rho^L_\kappa$ increases strictly in $\kappa\in[\kappa_c(L),1)$.
		
		{\it Step 2.}
		Let us prove that $\kappa_c(L)<1$ and $\rho^L_{\kappa_c(L)} = \rho$.
		For the sake of contradiction, 
		assume that $\kappa_c(L)=1$.
		Then by Step 1, for $\kappa\in[0,1)$, $\rho^L_\kappa=\rho$.
		As $\rho<1$, it contradicts \cref{:gen} \eqref{:gen:lim1}.
		Hence $\kappa_c(L)<1$.
		Now by \cref{:gen} \eqref{:gen:continuity} and Step 1, we have $\rho^L_{\kappa_c(L)} = \rho$.
		
		{\it Step 3.}
		Let us show that $P^L_{\kappa_c(L)}$ is rho-recurrent and $\kappa_c(L)>0$ when $L=\{x\}$.
		By \cref{:P:UkappaExpansion} \eqref{:P:UkappaExpansion:single}, 
		\begin{align}\label{:IS?:critical:pf:twoKappa}
			U^{\{x\}}_\kappa(x,x|1/\rho^{\{x\}}_\kappa) = \kappa/\rho^{\{x\}}_\kappa + (1-\kappa)U(x,x|1/\rho^{\{x\}}_\kappa), \quad \kappa\in[0,1).
		\end{align}
		It can be verified that the right hand side of \eqref{:IS?:critical:pf:twoKappa} is continuous in $\kappa$ by \cref{:gen} \eqref{:gen:monotonicity}, \eqref{:gen:continuity}, \cref{:P:U(1/rho)?1}, and the monotone convergence theorem.
		Hence, by Step 1, 2 and \cref{:P:U(1/rho)?1}, 
		$P^{\{x\}}_{\kappa_c({\{x\}})}$ is rho-recurrent.
		Therefore, $\kappa_c({\{x\}}) >0$ because otherwise it would contradict the condition that $P$ is rho-transient.
		
		{\it Step 4.}
		Let us show that $\kappa_c(L)>0$ when $L=\{x_1,\cdots,x_m\}$ with $m\ge 2$.
		Note that $\displaystyle P^L_{\kappa} = (\cdots((P)^{\{x_1\}}_\kappa)^{\{x_2\}}_\kappa\cdots)^{\{x_m\}}_\kappa$.
		By Step 1, 3 and \cref{:gen} \eqref{:gen:1}, we know that for any $x\in S$,
		\begin{statement}\label{:IS?:critical:pf:genQ}
			if $Q$ is a rho-transient irreducible transition probability on $S$ with spectral radius $<1$, then so is $Q^{\{x\}}_\kappa$ for some $\kappa\in(0,1)$.
		\end{statement}
		Now repeating using this, we can verify that 
		$\displaystyle (\cdots((P)^{\{x_1\}}_{\kappa_1})^{\{x_2\}}_{\kappa_2}\cdots)^{\{x_m\}}_{\kappa_m}$ is rho-transient for some $\kappa_1, \kappa_2, \cdots, \kappa_m \in (0,1)$.
		From \cref{:singleCri:3} \eqref{:singleCri:3:12}, we can verify that 
		$P^L_{\min\{\kappa_i: i=1,\cdots,m\}}$
		is rho-transient.
		Then $\kappa_c(L) \ge \min\{\kappa_i : i =1 ,\cdots, m\} > 0$ as desired.
		
		{\it Step 5.}
		Finally, let us show that $P^L_{\kappa_c(L)}$ is rho-recurrent when $L=\{x_1,\cdots,x_m\}$ with $m\ge 2$. 
		For the sake of contradiction, let us assume that $P^L_{\kappa_c(L)}$ is rho-transient.
		By Step 2, we know that $\rho^L_{\kappa_c(L)} = \rho <1$.
		From Steps 1,4 and \cref{:gen} \eqref{:gen:1}, we have that \eqref{:IS?:critical:pf:genQ} holds with $\{x\}$ being replaced by $L$. 
		Applying this to $P^L_{\kappa_c(L)}$, we know that $(P^L_{\kappa_c(L)})^L_\kappa$ is rho-transient for some $\kappa\in(0,1)$.
		This contradicts how $\kappa_c(L)$ is defined in Step 1.
		We are done.
	\end{proof}
	
	\begin{proof}[Proof of \cref{:IS?} \eqref{:IS?:finiteIS}]
		By $\rho<1$ and \cref{:gen} \eqref{:gen:1}, 
		$\rho^L_{\kappa}<1$ 
		for every $\kappa \in[0,1)$.
		As $P$ is rho-recurrent, by \cref{:singleCri:3} \eqref{:singleCri:3:12}, $P^L_\kappa$ is rho-recurrent for every $\kappa \in[0,1)$.
		Now by \cref{:singleCri:3} \eqref{:singleCri:3:3}, $\displaystyle \rho^L_\kappa < \rho^L_{\kappa'}$ for every $0\le\kappa<\kappa'<1$.
	\end{proof}
	
	\begin{proof}[Proof of \cref{:reach?} \eqref{:reach?:finiteUnder}]
		If 
		$L = \varnothing$ 
		, then $\rho^L_\kappa = \rho < \rho^S_\kappa$ by \eqref{:glLi}.
		If 
		$L \neq \varnothing$ 
		and $P^L_\kappa$ is rho-recurrent, then by taking $x\in S\setminus L$, \cref{:IS?} \eqref{:IS?:finiteIS} and \cref{:gen} \eqref{:gen:monotonicity},
		$\displaystyle \rho^L_\kappa < (\rho^L_\kappa)^{\{x\}}_\kappa \le \rho^S_\kappa$.
		If 
		$L \neq \varnothing$ 
		and $P^L_\kappa$ is rho-transient, then, by \cref{:singleCri:3} \eqref{:singleCri:3:12} and \eqref{:glLi}, $\rho^L_\kappa = \rho < \rho^S_\kappa$.
	\end{proof}
	
	\section{\texorpdfstring{Proof of \cref{:twoRhoRecurrences}, \cref{:reach?} \eqref{:reach?:invariant}, \eqref{:reach?:strictlyRhoRecurrentUnder} and \cref{:IS?} \eqref{:IS?:comFiniteIS}}{}}\label{:pf2}
	
	\begin{proof}[Proof of \cref{:twoRhoRecurrences}]
		{\it Step 1.} Assuming that $P$ is rho-recurrent, let us prove that $P^L_\kappa$ is strictly rho-recurrent for every $\kappa\in(0,1)$.
		By \cref{:singleCri:3} \eqref{:singleCri:3:12}, $P^L_\kappa$ is rho-recurrent. 
		Suppose that $P^L_\kappa$ is not strictly rho-recurrent.
		Thus for arbitrarily fixed $x\in L$, we have
		\begin{align}\label{:comFiReach:globallyUniformlyLazyMaintenance:=}
			r(U^L_\kappa|x,x) = 1/\rho^L_\kappa.
		\end{align}
		As $\kappa>0$, there exists $\kappa', \kappa_1 \in (0,\kappa)$ such that 
		$\displaystyle (( P^{L\setminus\{x\}}_{\kappa} )^{\{x\}}_{\kappa'})^{\{x\}}_{\kappa_1} = P^L_\kappa$.
		Let $\displaystyle P_{(1)}: = ( P^{L\setminus\{x\}}_{\kappa} )^{\{x\}}_{\kappa'}$.
		As $P$ is rho-recurrent, by \cref{:singleCri:3} \eqref{:singleCri:3:12}, $P_{(1)}$ is rho-recurrent.
		Denote by $\rho_{(1)}, U_{(1)}$ and $r(U_{(1)}|x,x)$ the spectral radius, U-function, and the convergence radius of the U-function of $P_{(1)}$, receptively.
		Then by \eqref{:P:rG<=rU}, \cref{:P:rU=rU}, \eqref{:comFiReach:globallyUniformlyLazyMaintenance:=} and \cref{:gen} \eqref{:gen:monotonicity}, 
		$1/\rho_{(1)} \le r(U_{(1)}|x,x) = r(U^L_\kappa|x,x) = 1/\rho^L_\kappa \le 1/\rho_{(1)}$.
		Thus we have $1/\rho_{(1)} = 1/\rho^L_\kappa$.
		Since $P^L_\kappa$ is rho-recurrent and $\displaystyle (P_{(1)})^{\{x\}}_{\kappa_1} = P^L_{\kappa}$, by \cref{:P:U(1/rho)?1} and \cref{:P:UkappaExpansion} \eqref{:P:UkappaExpansion:single}, 
		\begin{align}
			1 = U^L_\kappa(x,x|1/\rho^L_\kappa) = (U_{(1)})^{\{x\}}_{\kappa_1}(x,x|1/\rho^L_\kappa) = \kappa_1 /\rho^L_\kappa + (1-\kappa_1) U_{(1)}(x,x|1/\rho^L_\kappa)
		\end{align}
		where we denote by $\displaystyle (U_{(1)})^{\{x\}}_{\kappa_1}$ the U-function of the $(\{x\},\kappa_1)$-lazy version of $P_{(1)}$.
		By $\rho<1$ and \cref{:gen} \eqref{:gen:1}, $1/\rho^L_\kappa>1$.
		Together with $\kappa_1\in(0,1)$, we have $1 > U_{(1)}(x,x|1/\rho^L_\kappa) = U_{(1)}(x,x|1/\rho_{(1)})$.
		Now by \cref{:P:U(1/rho)?1}, $P_{(1)}$ is rho-transient, which is a contradiction.
		
		{\it Step 2.}
		Suppose that $P$ is rho-transient.
		By \cref{:IS?} \eqref{:IS?:critical}, $P^L_{\kappa_c(L)}$ is rho-recurrent.
		Replacing $P$ by $P^L_{\kappa_c(L)}$ in Step 1, we can verify that $P^L_\kappa$ is strictly rho-recurrent for $\kappa\in(\kappa_c(L),1)$.
		
		{\it Step 3.} 
		It remains to show that $P^L_{\kappa_c(L)}$ is critically rho-recurrent when $P$ is rho-transient.
		As $P^L_{\kappa_c(L)}$ is rho-recurrent (\cref{:IS?} \eqref{:IS?:critical}), we assume for the sake of contradiction that $P^L_{\kappa_c(L)}$ is strictly rho-recurrent.
		Then by \eqref{:P:SCR>}, there exists $x\in S$ such that 
		\begin{equation}\label{:twoRhoRecurrences:pf:rho2<rho}
			r(U^L_{\kappa_c(L)}|x,x) > 1/\rho^L_{\kappa_c(L)}.
		\end{equation}
		
		Case a): $P^L_{\kappa_c(L)}(x,x)>0$. 
		In this case, for every $\kappa_2\in[0,P^L_{\kappa_c(L)}(x,x)]$, observe that there exists transition probability $\tilde P(\kappa_2)$ such that $\displaystyle (\tilde P(\kappa_2))^{\{x\}}_{\kappa_2} = P^L_{\kappa_c(L)}$.
		Denote by $\tilde \rho(\kappa_2), \tilde U(\kappa_2)$ and $r(\tilde U(\kappa_2)|x,x)$ the spectral radius, the U-function, and the convergence radius of the U-function of $\tilde P(\kappa_2)$, respectively. 
		By \cref{:P:rU=rU}, $r(\tilde U(\kappa_2)|x,x) = r(U^L_{\kappa_c(L)}|x,x)$.
		Note that 
		$$\tilde P(\kappa_2) = (\tilde P(P^L_{\kappa_c(L)}(x,x)))^{\{x\}}_{\frac{P^L_{\kappa_c(L)}(x,x)-\kappa_2}{1-\kappa_2}}.$$
		Therefore, by \cref{:gen} \eqref{:gen:continuity}, $1/\tilde \rho(\kappa_2)$ is continuous in $\kappa_2$.
		As $\tilde \rho(0) = \rho^L_{\kappa_c(L)}$ and \eqref{:twoRhoRecurrences:pf:rho2<rho},
		we can choose some $\kappa_2\in(0,P^L_{\kappa_c(L)}(x,x))$ small enough, such that $1/\tilde \rho(\kappa_2) < r(\tilde U(\kappa_2)|x,x)$.
		Thus by \cref{:P:rhoTransient}, $\tilde P(\kappa_2)$ is rho-recurrent.
		As $\rho<1$, by \cref{:gen} \eqref{:gen:1}, 
		$\tilde \rho(\kappa_2)<1$. 
		Then by \cref{:singleCri:3} \eqref{:singleCri:3:3} and \cref{:IS?} \eqref{:IS?:critical},
		\begin{align}\label{rho2<rho}
			\tilde \rho(\kappa_2)<\rho^L_{\kappa_c(L)} = \rho.
		\end{align}
		
		This leads to the following contradiction.
		Observe that $\tilde P(\kappa_2)(y,y)>0$ for every $y\in L$ 
		(If $x\neq y$, then $\tilde P(\kappa_2)(y,y) = P^L_{\kappa_c(L)}(y,y)>0$; 
		If $x=y$, as $\kappa_2<P^L_{\kappa_c(L)}(x,x)$, $\tilde P(\kappa_2)(y,y)>0$ also holds). 
		Thus there exists $\kappa_3 \in (0,1)$ and transition probability $P_{(2)}$ such that
		$\displaystyle (P_{(2)})^L_{\kappa_3} = \tilde P(\kappa_2)$.
		Let $\rho_{(2)}$ be the spectral radius of $P_{(2)}$.
		Hence by \cref{:gen} \eqref{:gen:monotonicity}, we have that 
		\begin{align}\label{rho3<=rho2}
			\rho_{(2)} \le \tilde \rho(\kappa_2).
		\end{align}
		Observe that $\displaystyle P^L_{\kappa_c(L)} = ((P_{(2)})^L_{\kappa_3})^{\{x\}}_{\kappa_2} = ((P_{(2)})^{\{x\}}_{\kappa_2})^L_{\kappa_3}$.
		Using \cref{:IS?} \eqref{:IS?:critical}, we have that $\displaystyle (P_{(2)})^{\{x\}}_{\kappa_2}$ is rho-transient and its spectral radius is $\rho$.
		Then by \cref{:singleCri:3} \eqref{:singleCri:3:12}, $\rho_{(2)}=\rho$. 
		Together with \eqref{rho2<rho} and \eqref{rho3<=rho2} forms a contradiction.
		
		Case b): $P^L_{\kappa_c(L)}(x,x)=0$.
		Fix an arbitrary $\kappa\in[0,1)$ and define $Q:=P^S_\kappa$.
		By \cref{:P:globallyUniformlyLazyMaintenance} and our assumptions about $P$ and $P^L_{\kappa_c(L)}$, $Q$ is rho-transient and $\displaystyle Q^L_{\kappa_c(L)} = (P^L_{\kappa_c(L)})^S_\kappa$ is strictly rho-recurrent.
		By \eqref{:glLi} and \cref{:P:globallyUniformlyLazyMaintenance}, we can verify that $\kappa_c(L)$ is also the critical value in \cref{:IS?} \eqref{:IS?:critical} with respect to the underlying transition probability $Q$ and lazy state $L$.
		Now since $\displaystyle Q^L_{\kappa_c(L)}(y,y)>0$ holds for every 
		$y\in S$, 
		we can argue similarly as in case a) and arrive at a contradiction.
	\end{proof}
	
	\begin{proof}[Proof of \cref{:reach?} \eqref{:reach?:invariant}]
		By \cref{:P:globallyUniformlyLazyMaintenance}, if $P$ is rho-transient or critically rho-recurrent, so is $P^S_{\kappa}$.
		Note that $\displaystyle P^S_{\kappa} = (P^L_{\kappa})^{S\setminus L}_\kappa$.
		Thus $P^S_{\kappa}$ is the $(S\setminus L,\kappa)$-lazy version of $P^L_{\kappa}$.
		For $\kappa\in(0,1)$,
		if $P^S_{\kappa}$ is rho-transient or critically rho-recurrent, then by applying \cref{:twoRhoRecurrences} and \cref{:IS?} \eqref{:IS?:critical}, we have 
		$\rho^L_{\kappa}=\rho^S_{\kappa}$.
		If $P^S_{\kappa}$ is strictly rho-recurrent, 
		then by \cref{:twoRhoRecurrences} and \cref{:IS?} \eqref{:IS?:critical} and \eqref{:IS?:finiteIS},
		we have $\rho^L_{\kappa}<\rho^S_{\kappa}$.
	\end{proof}

	\begin{proof}[Proof of \cref{:reach?} \eqref{:reach?:strictlyRhoRecurrentUnder}]
		By \cref{:P:globallyUniformlyLazyMaintenance}, $P^S_\kappa$ is strictly rho-recurrent.
		As $\texttt{\#} (S\setminus L)>0$, we can choose $x\in S\setminus L$.
		By \cref{:gen} \eqref{:gen:monotonicity}, $\displaystyle \rho^L_\kappa \le \rho^{S \setminus \{x\}}_\kappa$.
		As $\displaystyle (P^{S \setminus \{x\}}_\kappa)^{\{x\}}_\kappa = P^S_\kappa$,
		by \cref{:twoRhoRecurrences} and \cref{:IS?} \eqref{:IS?:critical} and \eqref{:IS?:finiteIS}, $\displaystyle \rho^{S \setminus \{x\}}_\kappa < \rho^S_\kappa$.
		We are done.
	\end{proof}
	
	\begin{proof}[Proof of \cref{:IS?} \eqref{:IS?:comFiniteIS}]
		Take $x\in L$ and $0\le \kappa_1 < \kappa_2 <1$.
		Firstly assume that $P^L_{\kappa_1}$ is rho-recurrent.
		By $\rho<1$ and \cref{:gen} \eqref{:gen:1}, we have $\rho^L_{\kappa_1}<1$.
		By \cref{:singleCri:3} \eqref{:singleCri:3:3} and \cref{:gen} \eqref{:gen:monotonicity}, we have $\displaystyle \rho^L_{\kappa_1}< (\rho^L_{\kappa_1})^{\{x\}}_{\kappa} \le \rho^L_{\kappa_2}$ where $\kappa :=\frac{\kappa_2-\kappa_1}{1-\kappa_1}$.
		Now assume that $P^L_{\kappa_1}$ is rho-transient.
		By \cref{:P:globallyUniformlyLazyMaintenance}, $(P^L_{\kappa_1})^S_{\kappa}$ is also rho-transient.
		As $\displaystyle (P^L_{\kappa_1})^S_{\kappa} = (P^L_{\kappa_2})^{S\setminus L}_{\kappa}$, by \eqref{:glLi}, $\displaystyle (P^L_{\kappa_2})^{S\setminus L}_{\kappa}$ is rho-transient and $\displaystyle \rho^L_{\kappa_1} < (\rho^L_{\kappa_1})^S_{\kappa} = (\rho^L_{\kappa_2})^{S\setminus L}_{\kappa}$. 
		As $\texttt{\#} (S\setminus L) <\infty$, by \cref{:singleCri:3} \eqref{:singleCri:3:12}, $\displaystyle (\rho^L_{\kappa_2})^{S\setminus L}_{\kappa} = \rho^L_{\kappa_2}$.
		We are done. 
	\end{proof}

\end{document}